\documentclass[10pt, a4paper]{amsart}
\usepackage{amsfonts}
\usepackage{amssymb, amsmath}
\usepackage{url}
\usepackage{xcolor}
\usepackage{enumerate}

\newtheorem{theorem}{Theorem}[section]
\newtheorem{proposition}[theorem]{Proposition}
\newtheorem{corollary}[theorem]{Corollary}
\newtheorem{lemma}[theorem]{Lemma}

\theoremstyle{definition}
\newtheorem{definition}[theorem]{Definition}
\newtheorem{example}[theorem]{\textbf{Example}}
\newtheorem{remark}[theorem]{\textbf{Remark}}

\newtheorem{notation}[theorem]{\textbf{Notation}}

\title{Critical binomial ideals of Northcott type}

\author{{P.A.} {Garc\'{\i}a S\'anchez}}
\address{Departamento de \'Algebra, Universidad de Granada, E-18071 Granada, Espa\~na}
\email{pedro@ugr.es}

\author{{D.} Llena}
\address{Departamento de Matem\'{a}ticas, Universidad de Almeria, E-04120 Almeria,  Espa\~na}
\email{dllena@ual.es}

\author{I. Ojeda}
\address{Departamento de Matem\'{a}ticas, Universidad de Extremadura,  E-06071 Badajoz, Espa\~na}
\email{ojedamc@unex.es}

\thanks{The authors were supported by the project
MTM2017-84890-P, which is funded by Ministerio de
Econom\'{\i}a y Competitividad and Fondo Europeo de
Desarrollo Regional FEDER. The first and third authors
were also partially supported by the project
PGC2018-096446-B-C21 (MINECO/FEDER, UE).
The first and second authors were supported by the
Junta de Andaluc\'{\i}a Grant Number FQM-343,
and the third author was supported by the
Junta de Extremadura Grant Number FQM024-GR18021
(FEDER, UE).}



\begin{document}

\begin{abstract}
In this paper, we study a family of binomial ideals defining monomial curves in the $n-$dimensional affine space determined by $n$ hypersurfaces of the form $x_i^{c_i} - x_1^{u_{i1}} \cdots x_n^{u_{1n}} \in \Bbbk[x_1, \ldots, x_n]$ with $u_{ii} = 0$, $i\in \{ 1, \ldots, n\}$. We prove that the monomial curves in that family are set-theoretic complete intersections. Moreover, if the monomial curve is irreducible, we compute some invariants such as genus, type and Frobenius number of the corresponding numerical semigroup. We also describe a method to produce set-theoretic complete intersection semigroup ideals of arbitrary large height.
\end{abstract}

\maketitle

\medskip
\begingroup
\footnotesize

\noindent
This is the author-accepted manuscript of
P. A. Garc\'{\i}a-S\'anchez, D. Llena and I. Ojeda,
\emph{Critical binomial ideals of Northcott type},
\emph{Journal of the Australian Mathematical Society}
\textbf{110} (2021), no.~1, 48--70.

\smallskip
\noindent
Published version:
\url{https://doi.org/10.1017/S1446788720000385}

\smallskip
\noindent
This manuscript is made available under the
arXiv.org perpetual, non-exclusive license.
Further redistribution or reuse by third parties
requires permission from the rights holder.

\smallskip
\noindent
\textcopyright{} 2020 Australian Mathematical
Publishing Association Inc.

\endgroup
\medskip

\section{Introduction}\label{sec:intro}

We study a family of unit free monoids contained in $T\oplus \mathbb Z$, with $T$ a finite Abelian group. The monoids of this family have an associated ideal generated by critical binomials, and they extend several classes of numerical semigroups studied by Bresinsky, Alc\'antar and Villareal, Herzog, O'Carrol and Planas, among others \cite{Villarreal, Bresinsky, Herzog70, OCPV}. These monoids will be called semigroups of Northcott type, since their defining ideals are inspired in \cite{Northcott}. 

Ideals of Northcott type are defined as the ideal generated by the coordinates of the product of a matrix $\Phi$ and a vector $\mathbf{m}$. The matrix $\Phi$ has all its entries equal to zero except in the diagonal and below the diagonal (and the upper right corner); nonzero entries are monomials in a single variable. The vector $\mathbf{m}$ is also a vector of monomials in each of the variables. We show that every ideal of Northcott type is realizable as the defining ideal of a monoid of rank one. By gluing these monoids with copies of $\mathbb{N}$, we cover and extend the families studied in \cite{Villarreal}. 

We are able to prove that ideals of Northcott type have a unique (up to signs) minimal system of binomial generators. This in particular means that monoids of Northcott type are uniquely presented. We also depict how the primary decomposition of these ideals looks like. 

As monoids of Northcott type are uniquely presented, we are able to derive some nonunique factorization properties. We show what are the maximum and minimum of the Delta sets of factorizations of the monoid, as well as the catenary degree. 

As a particular case, we study the case when our monoids are numerical semigroups. We show how to find families of Northcott type ideals yielding numerical semigroups, and compute some relevant invariants in terms of the exponents of the monomials appearing in $\Phi$ and $\mathbf{m}$. We start by computing the Ap\'ery sets, and from these we easily derive formulas for the genus, Frobenius number and pseudoFrobenius numbers (and thus the type) of the semigroup.

We show that if we have a monoid with defining ideal a critical ideal, and we glue this monoid with a copy of $\mathbb{N}$, then the resulting monoid is again critical. As we mentioned above, this allows us to cover several families studied in the literature that appear as a gluing of a Northcott type monoid and a copy of $\mathbb N$. One can repeat this process and obtain a family of critical monoids. Not every critical monoid can be obtained in this way.

In the final section we show that every ideal of Northcott type is a set-theoretic complete intersection. As gluings by copies of $\mathbb{N}$ preserve set-theoretic complete intersections, the monoids of the family considered in this manuscript are set-theoretic complete intersections.

\section{Preliminaries}\label{sec:prelim}

Let $G$ be a finitely generated commutative group and set $\mathcal{A} = \{\mathbf{a}_1, \ldots, \mathbf a_n\} \subset G$. Let $\mathbb{N} \mathcal{A} := \mathbb{N} \mathbf a_1 + \dots + \mathbb{N} \mathbf a_n$ be the commutative monoid generated by $\mathcal{A}$, where $\mathbb{N}$ stands for the set of non-negative integers. Throughout this paper, we assume that $\mathcal{A}$ is chosen to have the following property $\mathbb{N} \mathcal{A} \cap (- \mathbb{N} \mathcal{A}) = \{0\}$, that is to say, $\mathbb{N} \mathcal{A}$ is supposed to be unit free. 

Let $\Bbbk[\mathbf{x}] := \Bbbk[x_1, \ldots, x_n]$ be the polynomial ring in $n$ variables over an arbitrary field $\Bbbk$. As usual, we will denote by $\mathbf{x}^\mathbf{u}$ the monomial $x_1^{u_1} \cdots x_n^{u_n}$ of $\Bbbk[\mathbf{x}] ,$ where $\mathbf{u} = (u_1, \ldots, u_n) \in \mathbb{N}^n$. The ring $\Bbbk[\mathbf{x}]$ can be graded by $\mathbb N \mathcal{A}$ by assigning degree $\mathbf a_i$ to the variable $x_i$, $i\in\{ 1, \ldots, n\}$. So, by abusing the notation slightly, we will say that $\Bbbk[\mathbf{x}]$ is $\mathcal{A}-$graded.

Let $\mathbf{t}$ be a symbol. The semigroup ring associated to $\mathbb{N}\mathcal{A}$ is $\Bbbk[\mathcal{A}] := \bigoplus\nolimits_{\mathbf{a} \in \mathbb{N} \mathcal{A}} \Bbbk \mathbf{t}^\mathbf{a}$. Its elements can be seen as polynomials in $\mathbf{t}$ with exponents in $\mathbb{N}\mathcal{A}$. Addition is performed component-wise, and the product is defined by distributivity and the rule $\mathbf{t}^\mathbf{a} \cdot \mathbf{t}^{\mathbf{a}'} = \mathbf{t}^{\mathbf{a} + \mathbf{a}'}$. Consider the ring homomorphism \[\pi : \Bbbk[\mathbf{x}] \to \Bbbk[\mathcal{A}];\ x_i \mapsto \mathbf{t}^{\mathbf{a}_i}.\] By \cite[Lemma 7.3]{MS05}, the kernel of $\pi$ is the ideal generated by all the binomials $\mathbf{x}^\mathbf{u} - \mathbf{x}^\mathbf{v}$ such that $\pi(\mathbf{x}^\mathbf{u}) = \pi(\mathbf{x}^\mathbf{v})$. This binomial ideal is usually called the semigroup ideal of $\mathcal{A}$ and it is denoted by $I_\mathcal{A}$. Clearly, $\Bbbk[\mathcal{A}]$ is $\mathcal{A}-$graded and $I_\mathcal{A}$ is $\mathcal{A}-$homogeneous. 

It is convenient to observe that $I_\mathcal{A}$ is not necessarily a prime ideal. In fact, $I_\mathcal{A}$ is prime if and only if $\mathbb{Z} \mathcal{A} := \mathbb{Z} \mathbf a_1 + \dots + \mathbb{Z} \mathbf a_n$ is torsion free (this will happen, in particular, when $G$ has not torsion). In this case, $I_\mathcal{A}$ is a so-called \textbf{toric ideal} and the monoid $\mathbb{N} \mathcal A$ is said to be an \textbf{affine semigroup}.

Recall that the dimension of the zero set of $I_\mathcal{A}$ is equal to the rank of $\mathbb Z \mathcal A$ (see, e.g., \cite[Proposition 7.5]{MS05}).

\begin{definition}
With the above notation, a binomial $x_i^{c_i} - \prod_{j \neq i} x_j^{u_{ij}} \in I_\mathcal{A}$ is called \textbf{critical with respect to $x_i$} if $c_i$ is the least positive integer such that $c_i \mathbf a_i \in \sum_{j \neq i} \mathbb{N} \mathbf a_j$. A subideal of $I_\mathcal{A}$ generated by $n$ critical binomials, one for each variable, is said to be a \textbf{critical binomial ideal associated to $\mathcal{A}$}.
\end{definition}

Notice that $c_i>1$ for all $i$ implies that $\mathcal{A}$ is a minimal generating system of $\mathbb{N}\mathcal{A}$, in the sense that no proper subset of $\mathcal{A}$ forms a generating set of $\mathbb{N}\mathcal{A}$.

Our first result gives a sufficient condition for the existence of critical binomial ideals.

\begin{proposition}\label{Prop1}
If $\mathbb{Z}\mathcal{A}$ has rank $1$, then $I_\mathcal{A}$ contains a critical binomial ideal associated to $\mathcal{A}$.
\end{proposition}

\begin{proof}
By the fundamental theorem of finitely generated commutative groups $\mathbb{Z}\mathcal{A} \cong \mathbb{Z} \oplus T$ for some finite (abelian) group $T$. Let $t \geq 0$ be the order of $T$.
Thus, we may suppose that $\mathbf{a}_i = (a_i,\mathbf{b}_i)$ such that $a_i \in \mathbb{Z}$ and $\mathbf b_i \in T$, $i\in\{ 1, \ldots, n\}$. Clearly, $a_1 a_i \in \sum_{j \neq i} \mathbb N a_j,$  $i \in\{ 2, \ldots, n\}$ and $a_n a_1 \in \sum_{j \neq 1} \mathbb N a_j$, so it follows that $(t a_1) \mathbf a_i \in \sum_{j \neq i} \mathbb N \mathbf a_j$, $i \in\{ 2, \ldots, n\}$ and $(t a_n) \mathbf a_1 \in \sum_{j \neq 1} \mathbb N \mathbf a_j$. Therefore, for each $i$, there exists a minimal $c_i > 0$ such that $c_i \mathbf{a}_i \in \sum_{j \neq i} \mathbb N \mathbf a_j$. Thus, $I_\mathcal{A}$ contains a critical binomial ideal.
\end{proof}

The converse of the above proposition is not true in general.

\begin{example}\label{example1}
Let $\mathbf{a}_i$ be the $i$th column of the following matrix 
\[A = \left(\begin{array}{ccccc} 3 & 5 & 7 & 0 & 0 \\ 0 & 0 & 0 & 2 & 3 \end{array}\right) \] and set $\mathcal{A} = \{\mathbf{a}_1, \ldots, \mathbf{a}_5\} \subseteq \mathbb{Z}^2$.
By direct computation, one can check that $I_\mathcal{A}$ is generated by critical binomials.
\end{example}

\begin{definition}
With the above notation, we will say that $\mathbb N\mathcal A$ is a \textbf{critical monoid} if $I_\mathcal A$ is a critical binomial ideal associated to $\mathcal{A}$. In this case, for short, we will say that the binomial ideal $I_\mathcal A$ is \textbf{critical}.
\end{definition}

In the following, we will assume that $\mathbb Z \mathcal A$ has rank one; in particular $I_\mathcal{A}$ will contain a critical binomial ideal associated to $\mathcal{A}$ by Proposition \ref{Prop1}. One of the aims of this work is to give conditions for $\mathbb N\mathcal A$ to be a critical monoid; equivalently, for $I_\mathcal{A}$ to be critical. For this purpose, we summarize first the known results:

\medskip
If $\mathcal{A} = \{\mathbf a_1, \ldots, \mathbf a_n\} \subseteq \mathbb{N}$ and $\mathrm{gcd}(\mathbf a_1, \ldots,\mathbf a_n) = 1$, then $\mathbb{N}\mathcal{A}$ is a numerical semigroup and $I_\mathcal{A}$ is the defining ideal of an irreducible monomial curve. In this case, it is known that every minimal system of binomial generators of $I_\mathcal{A}$ contains, up to sign, a critical binomial with respect to $x_i,$ for each $i$ (see, for example, \cite[Proposition 2.3]{KaOj}).

Under the hypothesis of $\mathbb{N}\mathcal{A}$ being a numerical semigroup, Herzog proved that every (toric) ideal $I_\mathcal{A}$ in $\Bbbk[x_1, x_2, x_3]$ is critical (see \cite[Section 3]{Herzog70}, or \cite[Theorem 2.2]{OjPis} for a concise formulation of Herzog's results). For $n = 4,$ Alc\'antar and Villarreal proved in \cite{Villarreal} that if $I_\mathcal{A}$ is critical of height $3,$ then (after permuting the variables appropriately) it has a set of binomial generators of one of the following three types.
\begin{itemize}
\item[(N):] $f_1 = x_1^{c_1} - x_3^{u_{13}} x_4^{u_{14}}, f_2 = x_2^{c_2} - x_1^{u_{21}} x_4^{x_{24}}, f_3 = x_3^{c_3} - x_2^{u_{32}} x_4^{u_{34}}$ and $D = x_4^{c_4} - x_1^{u_{41}} x_2^{u_{42}} x_3^{u_{43}},$ with $u_{ij} > 0,$ for all $i,j$.
\item[(gN1):]  $f_1 = x_1^{c_1} - x_3^{c_3}, f_2 = x_2^{c_2} - x_1^{u_{21}} x_3^{u_{23}}$ and $g = x_4^{c_4} - x_1^{u_{41}} x_2^{u_{42}} x_3^{u_{43}}$.
\item[(gN2):]  $f_1 = x_1^{c_1} - x_2^{u_{12}} x_3^{u_{13}}, f_2 = x_2^{c_2} - x_1^{u_{21}} x_3^{u_{23}}, D = x_3^{c_3} - x_1^{u_{31}} x_2^{u_{32}}$ and $g =  x_4^{c_4} - x_1^{u_{41}} x_2^{u_{42}} x_3^{u_{43}},$ with $u_{ij} > 0$ for all $i \neq 4$ and $u_{4j} \geq 0,$ not all zero, for all $j$.
\end{itemize}
Moreover, under suitable arithmetical conditions on the exponents, the converse turns out to be true (see Propositions \ref{Prop Th5Bresinsky} and \ref{prop:gN*}).

The following result due to Bresinsky (\cite[Theorem 5]{Bresinsky}) gives a necessary and sufficient condition for the critical ideal generated by (N) to be equal to $I_\mathcal{A}$. 

\begin{proposition}\label{Prop Th5Bresinsky}
With the same notation as in (N), let $B$ be the matrix whose rows are the exponent vectors of $f_1, f_2, f_3$ and $D$, that is,  
$$
B = \left(\begin{array}{cccc}
c_1 & 0 & -u_{13} & -u_{14} \\
-u_{21} & c_2 & 0 & -u_{24} \\
0 & -u_{32} & c_3 &  -u_{34} \\
-u_{41} &  -u_{42} & -u_{43} & c_4
\end{array}\right).
$$
Then there exists $\mathcal{A} \subset \mathbb{N}$  such that $\mathbb N \mathcal A$ is a numerical semigroup and $I_\mathcal{A} = \langle f_1, f_2, f_3, D \rangle$ if and only if
\begin{itemize}
\item[(a)] the sum of the columns of $B$ is zero;
\item[(b)] the $(4,j)-$minors of $B$, $j \in\{ 1, \ldots, 4\}$ are relatively prime.
\end{itemize}
\end{proposition}

See \cite[Proposition 3.2]{Villarreal} for a proof in our notation.

Notice that conditions (a) and (b) in the above proposition imply that $B$ has rank $3$. So, there exist $a_1, \ldots, a_4 \in \mathbb N$ relatively prime such that $(a_1, \ldots, a_4) B = 0$. In particular, we have that the $(4,j)-$minor of $B$ is equal to $\alpha a_j$ for some nonzero integer $\alpha$ which is necessarily equal to one by condition (b).

\section{Critical ideals of Northcott type}\label{sec:critical-ideals-nt}

In this section, we will generalize the family (N) for arbitrary $n$ by producing a family of critical binomial ideal of Northcott type of the form $I_\mathcal{A}$. This will allow us to determine whole families of numerical semigroups whose semigroup ideal associated is critical.

It is fair to point out that the case $n=3$ has been analyzed exhaustively by O'Carroll and Planas-Vilanova in \cite{OCPV} for a Noetherian ring $R$ and $\mathbf m = (x_1, \ldots, x_n)$ being a sequence of elements of $R$ generating an ideal of height $n$. 

Now, for convenience, we reproduce verbatim (but adapted to our notation) the definition of Northcott ideal given in \cite{OCPV}: ``In \cite{Northcott}, Northcott considered the following situation: let $\mathbf{m} =
(m_1 , \ldots, m_d)$ and $\mathbf{f} = (f_1 ,  \ldots,  f_d)$ be two sets of $d$ elements of a Noetherian ring $R,$ connected
by the relations
\[
\begin{array}{ccc}
f_1 & = & a_{11} m_1 + a_{12} m_2 + \ldots + a_{1d} m_d,\\
f_2 & = & a_{21} m_1 + a_{22} m_2 + \ldots + a_{2d} m_d,\\
& \vdots & \\
f_d & = & a_{d1} m_1 + a_{d2} m_2 + \ldots + a_{dd} m_d,
\end{array}
\]
with $a_{ij} \in R$. Let $D$ stand for the determinant of the $d \times d$ matrix $\Phi = (a_{ij})$. Northcott proved that if $\langle f_1 , \ldots, f_d \rangle$ has grade $d$ and $\langle f_1 , \ldots, f_d, D \rangle$ is proper, then the projective dimension of $R/\langle f_1 , \ldots, f_d, D \rangle$ is $d,$ and $\langle f_1 , \ldots , f_d , D \rangle$ and all its associated prime ideals have grade $d$ \cite[Theorem 2]{Northcott}. Subsequently, Vasconcelos called such an ideal $\langle f_1 , \ldots, f_d , D \rangle$ the \textbf{Northcott ideal} associated to $\Phi$ and $\mathbf{m}$ (see, for example, \cite[p. 100]{Vasconcelos}).''

Returning to our problem of interest, one can observe that, by taking 
\[
\Phi = \left(\begin{array}{ccc}
x_1^{u_{41}} & 0 & -x_4^{u_{14}} \\
- x_4^{u_{24}} & x_2^{u_{42}} & 0 \\
0 & -x_4^{u_{34}} & x_3^{u_{43}} 
\end{array}\right)
\]
and $\mathbf{m} = (x_1^{u_{21}}, x_2^{u_{32}}, x_3^{u_{13}}),$
we have that $\Phi \, \mathbf{m}^\top = (f_1, f_2, f_3)^\top$; $\Phi \,  
(x_1^{u_{21}}, x_2^{u_{32}}, x_3^{u_{13}})^\top = (x_1^{u_{41}+u_{21}} - x_3^{u_{13}} x_4^{u_{14}}, x_2^{u_{42}+u_{32}} - x_1^{u_{21}} x_4^{x_{24}}, x_3^{u_{43}+u_{13}} - x_2^{u_{32}} x_4^{u_{34}})$ and $D = \det(\Phi) = x_4^{u_{14}+u_{24}+u_{34}} - x_1^{u_{41}} x_2^{u_{42}} x_3^{u_{43}}$. In other words, \emph{a toric ideal generated by (N) can be viewed as a Northcott type ideal} by Proposition \ref{Prop Th5Bresinsky}.

Clearly, we can generalize this situation in the following way.

\begin{definition}\label{Def CINT}
Let 
\begin{equation}\label{ecPhi}
\Phi = \left(\begin{array}{cccc}
x_1^{u_{n1}} & 0 & \ldots & - x_n^{u_{1n}} \\
- x_n^{u_{2n}} & x_2^{u_{n2}} & \ldots  & 0 \\
0 & - x_n^{u_{3n}} & \cdots & 0 \\
\vdots & \vdots & \ddots & \vdots \\
0 & 0 & \ldots & x_{n-1}^{u_{n n-1}}
\end{array}\right)
\end{equation}
and $\mathbf{m} = (x_1^{u_{21}}, \ldots, x_{n-2}^{u_{n-1 n-2}}, x_{n-1}^{u_{1 n-1}}),$ with $u_{ij} > 0,$ for all $i,j$. If $(f_1, \ldots, f_{n-1})^\top = \Phi\, \mathbf{m}^\top$ and $D = \det(\Phi),$ then the ideal of $\Bbbk[\mathbf{x}]$ generated by $\{f_1, \ldots, f_{n-1}, D\}$  will be called a \textbf{critical ideal of Northcott type}.
\end{definition}

\begin{remark}
Observe that we can extend the definition of critical ideal of the Northcott type to any  permutation of the variables.
However, for the sake of simplicity, throughout the document we will always assume that we are in the situation of Definition \ref{Def CINT}.
\end{remark}

Let us see that every critical ideal of Northcott type is the ideal associated to a (non-necessarily torsion-free) monoid. First, we need some preliminary notation.

\begin{notation}\label{not:gen-N}
To each binomial $\mathbf x^{\mathbf{v}}-\mathbf x^{\mathbf{v}'} \in \Bbbk[\mathbf x]$ we associate the vector $\mathbf{v}-\mathbf{v}' \in \mathbb Z^n$. With this notation, the rows of the following matrix 
\[ 
M := 
\begin{pmatrix}
u_{n1}+u_{21} & 0 & \ldots & 0 & -u_{1\, n-1} & -u_{1n} \\
-u_{21} & u_{n2}+u_{32} & \ldots & 0 & 0 & -u_{2n} \\
0 & -u_{32} &  \ldots & 0 & 0 & -u_{3n} \\
\vdots & \vdots & & \vdots & \vdots & \vdots \\
0 & 0 & \ldots & -u_{n-1\, n-2} & u_{n\, n-1}+u_{1\, n-1} & -u_{n-1\, n} \\ -u_{n1} & -u_{n2} & \ldots & -u_{n\, n-2} & -u_{n\, n-1}  & \sum_{i=1}^{n-1} u_{in}   
\end{pmatrix}
\]
correspond to the exponent vectors of the binomials $f_1, \ldots, f_{n-1}$ and $D$ in Definition \ref{Def CINT}.
\end{notation}

\begin{theorem}\label{th:gen-N}
With the same notation and assumptions as in Definition \ref{Def CINT}, let $ J = \langle f_1, \ldots, f_{n-1}, D \rangle$ be a critical ideal of Northcott type. Then there exist a commutative finite group $T$ and $\mathcal{A} \subset T \oplus \mathbb{N}$ such that $I_\mathcal{A} = J$. In particular, $\mathbb Z \mathcal A$ has rank one.
\end{theorem}

\begin{proof}
Let $M$ be the matrix corresponding to the exponent vectors of the binomials $f_1, \ldots, f_{n-1}$ and $D$ (see Notation \ref{not:gen-N}). Since the last row of $M$ is the sum of the rows above and the $(n-1)\times(n-1)$ minor obtained by removing the last row and column of $M$ is 
\[
\left(\prod\nolimits_{i=1}^{n-2} (u_{n i} + u_{i+1\, i}) \right)\left(u_{n\, n-1}+u_{1\, n-1}\right) - \left(\prod\nolimits_{i=1}^{n-2}  u_{i+1\, i} \right) u_{1\, n-1} > 0,
\]
we have that $M$ has rank $n-1$. Therefore, there exist unimodular matrices $P$ and $Q$ such that $P^{-1} M Q = \Delta$ where $\Delta$ is the Smith normal form of $M$, that is, $\Delta$ is a diagonal matrix with entries in the diagonal equal to $d_1, \ldots, d_n$ where $d_1 \mid d_2 \mid \cdots \mid d_{n-1}$ are positive integers and $d_n = 0$. It is well-known (see, for instance, \cite[Chapter 3]{fg}) that if $i$ is the smallest index such that $d_i \neq 1$ and $A$ is the $(n-i+1) \times n$ matrix whose rows are the last $n-i$ columns of $Q$, the kernel $\mathcal{M}$ of the map \[\mathbb{Z}^n \to \mathbb{Z}_{d_{i+1}} \oplus \cdots \oplus \mathbb{Z}_{d_{n-1}} \oplus \mathbb{Z};\ \mathbf u \mapsto A \mathbf u\] is generated by the rows of $M$. 

Let $\mathbf a_i = A\, \mathbf{e}_i$, where $\mathbf{e}_i \in  \mathbb{Z}^n$ is the vector that has a one in the $i$th coordinate and zeroes elsewhere, $i \in\{ 1, \ldots, n\}$, and set $\mathcal A := \{\mathbf{a}_1, \ldots, \mathbf{a}_n\}$.

Let us see that $\mathbb N \mathcal A \subseteq \mathbb{Z}_{d_{i+1}} \oplus \cdots \oplus \mathbb{Z}_{d_{n-1}} \oplus \mathbb{Z}$ is unit free. If there are $\mathbf a, \mathbf b \in \mathbb N\mathcal A$ such that $\mathbf a + \mathbf b=0$, then there exist $\mathbf v, \mathbf v' \in \mathbb N^n$ such that $\mathbf a=A\mathbf v$ and $\mathbf b = A \mathbf v'$. Hence $A(\mathbf v + \mathbf v')=0= A0$, which in particular implies that $\mathbf v + \mathbf v' \in \mathcal M\cap \mathbb N^n$. So, in order to prove that $\mathbb N\mathcal A$ is unit free, it suffices to show that $\mathcal M\cap\mathbb N^n=\{0\}$. And, as the last row of $M$ is the sum of the others, this is equivalent to prove that $(\alpha_1, \ldots, \alpha_{n-1}, 0) M \in \mathbb{N}^n$ implies $\alpha_i = 0$ for all $i$. The condition $(\alpha_1, \ldots, \alpha_{n-1}, 0) M \in \mathbb{N}^n$ is defined by the following system of inequalities:
\begin{align*}
\alpha_1 u_{n1} + (\alpha_1 - \alpha_2) u_{21} & \geq 0,\\
& \vdots \\
\alpha_{n-1} u_{n\, n-1} + (\alpha_{n-1} - \alpha_1) u_{1\, n-1} &\geq 0,\\
\sum\nolimits_{i=1}^{n-1}( - \alpha_i) u_{in} &\geq 0.
\end{align*}
If $\alpha_i \neq 0$ for some $i$, from the last inequality we obtain that $\alpha_j < 0$ for some $j$. For simplicity, we suppose $j = 1$. Thus, from the first inequality follows that $\alpha_2 < \alpha_1$. By repeating this argument with the second inequality, the third inequality, \ldots, and the $(n-1)$th inequality, we conclude that $\alpha_1 < \alpha_{n-1} < \cdots < \alpha_2 < \alpha_1 < 0$ which is impossible. Therefore $\alpha_i = 0$ for every $i$.

With this we have that $\mathbb N\mathcal A$ is a unit free cancellative monoid, and that $J\subseteq I_{\mathcal A}$. 

Now, as $\mathbb N \mathcal A$ is unit free, by \cite[Proposition 3.10]{fg}, we may suppose that $\mathcal{A} \subset \mathbb{Z}_{d_{i+1}} \oplus \cdots \oplus \mathbb{Z}_{d_n} \oplus \mathbb{N}$. Let $\mathbf{w} = (w_1, \ldots, w_n) \in \mathbb{N}^n$ be last row of $A$. Clearly the ideal $I_\mathcal{A}$ is homogeneous with respect to $\mathbf{w}$. 

Let $\prec$ be the monomial order on $\Bbbk[\mathbf{x}]$ represented by a matrix whose first row is $\mathbf{w}$ and second row is $(0, \ldots, 0, -1)$. We claim that $\mathcal G = \{f_1, \ldots, f_{n-1}, D\}$ is a Gr\"obner basis with respect to $\prec$. 

Since $\gcd\big(\mathrm{in}_\prec(f_i), \mathrm{in}_\prec(f_j) \big) = 1,$ it suffices to check that the $S$-polynomial $S(f_i, D)$ reduces to zero with respect to $\mathcal{G}$, for every $i\in \{ 1, \ldots, n-1\}$. By direct computation, one can see that 
\begin{align*}
S(f_i, D) = &\  x_{\sigma^{n-2}(i)}^{u_{n\, \sigma^{n-2}(i)}} \cdots  x_{\sigma^{2}(i)}^{u_{n\, \sigma^{2}(i)}} x_n^{u_{in}}  f_{\sigma(i)} + \\ & + x_{\sigma^{n-2}(i)}^{u_{n\, \sigma^{n-2}(i)}} \cdots  x_{\sigma^{3}(i)}^{u_{n\, \sigma^{3}(i)}}  x_n^{u_{in}+u_{\sigma(i) n}} f_{\sigma^2(i)} + \\ & + \cdots +  \\ & + x_n^{u_{in}+ \cdots + u_{\sigma^{n-3}(i) n}} f_{\sigma^{n-2}(i)}
\end{align*}
for every $i \in\{ 1, \ldots, n-1\}$, where $\sigma$ is the permutation $ n-1 \to n-2 \to \cdots \to 2 \to 1 \to n-1$. Therefore our claim follows and, by \cite[Theorem 3.1]{Bigatti}, we have that $\{f_1,\ldots, f_{n-1},D\}$ is a Gr\"obner basis of $(J:x_n^\infty)$, and consequently $ (J:x_n^\infty) = J$. 
Let $i \in \{1, \ldots, n-1\}$ and let $s_i$ be the smallest integer such that $(J:x_i^{s_i}) = (J:x_i^k),$ for some (every) $k \geq s_i$; clearly, $s_i$ exists by the Noetherian property of $\Bbbk[\mathbf x]$. In this case, one has that $J = (J:x_i^{s_i}) \cap (J+\langle x_i^{s_i} \rangle)$ (see, for instance \cite[Proposition ~7.2(a)]{Eisenbud96}). Moreover, taking into account that $D = x_1^{u_{n1}} \cdots x_{n-1}^{u_{n\, n-1}} - x_n^{\sum_{i=1}^{n-1} u_{in}} \in J$, we have that $x_i^{s_i-u_{ni}} x_n^{\sum_{i=1}^{n-1} u_{in}} \in  J+\langle x_i^{s_i} \rangle$ if $s_i \geq u_{ni}$ or $ x_n^{\sum_{i=1}^{n-1} u_{in}} \in J+\langle x_i^{s_i} \rangle $ if $s_i < u_{ni}$. Therefore, putting all this together, either 
\begin{align*}
(J:x_i^{s_i-u_{ni}}) & = (J:x_i^{s_i-u_{ni}} x_n^{\sum_{i=1}^{n-1} u_{in}})\\ & = \left( (J:x_i^{s_i}) \cap (J+\langle x_i^{s_i} \rangle) : x_i^{s_i-u_{ni}} x_n^{\sum_{i=1}^{n-1} u_{in}} \right) 
\\ & = \left( (J:x_i^{s_i}) : x_i^{s_i-u_{ni}} x_n^{\sum_{i=1}^{n-1} u_{in}} \right) \cap \left((J+\langle x_i^{s_i} \rangle) : x_i^{s_i-u_{ni}} x_n^{\sum_{i=1}^{n-1} u_{in}} \right)\\ & = (J:x_i^{2 s_i - u_{ni}}),
\end{align*}
in contradiction with the minimality of $s_i$, or 
\begin{align*} 
J & = (J:x_n^{\sum_{i=1}^{n-1} u_{in}}) = \left( (J:x_i^{s_i}) \cap (J+\langle x_i^{s_i} \rangle) : x_n^{\sum_{i=1}^{n-1} u_{in}} \right) 
\\ & = \left( (J:x_i^{s_i}) : x_n^{\sum_{i=1}^{n-1} u_{in}} \right) \cap \left((J+\langle x_i^{s_i} \rangle) : x_n^{\sum_{i=1}^{n-1} u_{in}}\right) = (J:x_i^{s_i})
\end{align*}

Therefore, we have that $J = ( (\ldots (J:x_1^{s_1}) : \ldots) : x_n^{s_n}) = ( J:(x_1 \cdots x_n)^s ),$ for some $s$. And, by \cite[Corollary 2.5]{Eisenbud96}, we conclude that $(J:(x_1\cdots x_n)^\infty) = I_\mathcal{A}$, whence $J=I_\mathcal A$.

The last statement is trivial.
\end{proof}

\begin{definition}
Let $\mathcal{A} \subseteq \mathbb Z \oplus T$ where $T$ denotes a commutative finite group. We will say that $\mathbb N \mathcal A$ is a \textbf{monoid of Northcott type} if $I_{\mathcal A}$ is a critical ideal of Northcott type.
\end{definition}

Observe that Theorem \ref{th:gen-N} establishes a bijective correspondence between monoid of Northcott type and critical ideal of Northcott type.   

Recall that a monoid $\mathbb N \mathcal A$ is \textbf{uniquely presented} if $I_\mathcal A$ admits  a unique minimal generating set of binomials (up to signs). Uniquely presented finitely generated commutative monoids where introduced in \cite{GSOj}, as the semigroup counterpart of the toric ideals generated by indispensable binomials (see \cite{OjVi2}).

\begin{corollary}\label{cor:uniq-pres}
Semigroups of Northcott type are uniquely presented. 
\end{corollary}

\begin{proof}
Let us see that every critical ideal of Northcott type has a unique minimal system of binomial generators (up to signs). 
Let \[J =  \langle x_i^{c_i} - \mathbf{x}^{\mathbf{u}_i} \mid \mathbf{u}_i \in \mathbb{N}^n,\ i \in\{1, \ldots, n\} \rangle\] be a critical ideal of Northcott type.
By Theorem \ref{th:gen-N}, there exists a finite subset of a group, $\mathcal{A} = \{\mathbf{a}_1, \ldots, \mathbf{a}_n\},$ such that $\mathbb Z \mathcal A$ has rank $1$ and $I_\mathcal{A} = J$. If $c_i \mathbf{a}_i \neq c_j \mathbf{a}_j$ for every $i \neq j$, by \cite[Corollary 5]{GSOj}, then $I_\mathcal{A}$ is uniquely generated and we are done. Otherwise, $\mathbf{x}^{\mathbf{u}_i} - \mathbf{x}^{\mathbf{u}_j} \in I_\mathcal{A}$ for some $i \neq j$. By definition, the supports of $\mathbf{u}_i$ and $\mathbf{u}_j$ are not disjoint. So, there exists $\alpha_k \in \mathbb{N}$ such that $\mathbf{x}^{\mathbf{u}_i} - \mathbf{x}^{\mathbf{u}_j} = x_k^{\alpha_k}(\mathbf{x}^{\mathbf{u}'_i} - \mathbf{x}^{\mathbf{u}'_j})$. As $\mathbb N\mathcal A$ is cancellative, $\mathbf{x}^{\mathbf{u}'_i} - \mathbf{x}^{\mathbf{u}'_j} \in I_\mathcal{A}$. But this is a contradiction with $\{x_i^{c_i} - \mathbf{x}^{\mathbf{u}_i} \mid \mathbf{u}_i \in \mathbb{N}^n\!, i \in \{1, \ldots, n\}\}$ being a Gr\"obner basis, as it was demonstrated in the last part of the proof of Theorem \ref{th:gen-N}.
\end{proof}

\subsection*{On the primary decomposition of critical ideals of Northcott type}\mbox{}\par

\medskip
Let $\Phi$ and $\mathbf m$ be as in Definition \ref{Def CINT}, and let $J$ be the corresponding critical ideal of Northcott type. By Theorem \ref{th:gen-N}, there exists a subset $\mathcal{A} = \{\mathbf{a}_1, \ldots, \mathbf{a}_n\}$ of a finitely generated commutative group such that $\mathbb Z \mathcal A$ has rank $1$ and $J = I_\mathcal{A}$.

Let $\mathcal{M}$ be the subgroup of $\mathbb{Z}^n$ generated by the rows of the matrix $M$ in Notation \ref{not:gen-N} and $\rho : \mathcal{M} \to \Bbbk^*$ be the null group homomorphism, that is, $\rho(\mathbf{v}) = 1,$ for all $\mathbf{v} \in \mathcal{M}$. Then $$I_\mathcal{A} = I_+(\rho) := \big\langle \{\mathbf{x}^{\mathbf{v}_+} - \rho(\mathbf{v}) \mathbf{x}^{\mathbf{v}_-}\ \mid\ \mathbf{v} \in \mathcal{M}\} \big\rangle,$$ where $\mathbf{v}_+$ and $\mathbf{v}_- \in \mathbb{N}^n$ denote the positive part and negative part of $\mathbf{v}$, respectively. 

Now, by simply studying the group structure of the finite group $\mathbb{Z}^n/\mathcal{M}$, one can explicitly describe
the radical, the associated primes and a minimal primary decomposition of a lattice ideal using the method developed by \cite{Eisenbud96} when $\Bbbk$ is algebraically closed.

We summarize the mentioned descriptions as in \cite{OjPie1}.

\begin{definition}
If $\mathcal{L}$ is a subgroup of $\mathbb{Z}^n,$ then the saturation of $\mathcal{L}$ is $$ \mathrm{Sat}(\mathcal{L}) := \{ \mathbf{v} \in \mathbb{Z}^n \mid d\, \mathbf{v} \in \mathcal{L} \mbox{ for some } d \in \mathcal{L} \setminus \{0\} \}. $$ We say that $\mathcal{L}$ is saturated if $\mathcal{L}=\mathrm{Sat}(\mathcal{L}).$
\end{definition}

Note that for any subgroup $\mathcal{L}$ of $\mathbb Z^n$, the group $\mathrm{Sat}(\mathcal{L})/\mathcal{L}$ is finite and that $\mathbb{Z}^n/\mathrm{Sat}(\mathcal{L})$ is torsion free. In particular, $\mathbb{Z}^n/\mathrm{Sat}(\mathcal{M}) \cong \mathbb{Z}$. 

The next Proposition is a consequence of Theorem 2.1(c) in \cite{Eisenbud96}.

\begin{proposition}\label{Prop PrimeSat}
The ideal $I_+(\rho)$ is prime if and only if $\mathcal{M}$ is saturated.
\end{proposition}

\begin{definition}
If $\mathcal{L}$ is a subgroup of $\mathbb{Z}^n$ and $p$ is a prime number, we define $\mathrm{Sat}_p(\mathcal{L})$ and ${\rm Sat'}_p(\mathcal{L})$ to be the largest sublattices of $\mathrm{Sat}(\mathcal{L})$
containing $\mathcal{L}$ such that $\mathrm{Sat}_p(\mathcal{L})/\mathcal{L}$ has order a power of
$p$ and $\mathrm{Sat'}_p(\mathcal{L})/\mathcal{L}$ has order relatively prime to $p.$ If
$p=0,$ we adopt the convention that $\mathrm{Sat}_p(\mathcal{L})=\mathcal{L}$ and ${\rm
Sat'}_p(\mathcal{L})=\mathrm{Sat}(\mathcal{L}).$
\end{definition}

The following result follows from Corollaries 2.2 and 2.5 in \cite{Eisenbud96}.

\begin{theorem}\label{Corolario 2.2}
Assume $\Bbbk$ algebraically closed and $\mathrm{char}(\Bbbk)=p \geq 0.$  Write $g$ for the order of $\mathrm{Sat'}_p(\mathcal{M})/\mathcal{M}.$ There are $g$ distinct group homomorphisms $\rho_1, \ldots, \rho_g$ from $\mathrm{Sat'}_p(\mathcal{M})$ to $\Bbbk^*$ extending $\rho$ and, for each $j$ a homomorphism $\rho'_j$ from $\mathrm{Sat}(\mathcal{M})$ to $\Bbbk^*$ extending $\rho_j.$ There is a unique group homomorphism $\rho'$ from $\mathrm{Sat}_p(\mathcal{M})$ to $\Bbbk^*$ extending $\rho.$ The radical, associated primes and minimal primary decomposition of
$I_\mathcal{A} = I_+(\rho)$ are: $$ \sqrt{I_+(\rho)}=I_+(\rho'), $$ $$
\mathrm{Ass}(S/I_+(\rho))= \{ I_+(\rho'_j) \mid j=1,\ldots, g \} $$
and $$ I_+(\rho) = \bigcap_{j=1}^g I_+(\rho_j) $$ where
$I_+(\rho_j)$ is $I_+(\rho'_j)$-primary. In particular, if $p=0,$
then $I_+(\rho)$ is a radical ideal. The associated primes
$I_+(\rho'_j)$ of $I_+(\rho)$ are all minimal and have the same
codimension $\mathrm{rank}(L_\rho).$
\end{theorem}


In the light of the theorem above, it suffices to know the
extensions of $\rho$ to $\mathrm{Sat}_p(\mathcal{M}),$ $\mathrm{Sat}(\mathcal{M})$ and $\ \mathrm{Sat'}_p(\mathcal{M})$ to find the radical, associated primes and the minimal primary decomposition of a lattice ideal, respectively.

\begin{example}
Let $$\Phi = \left(\begin{array}{cccc}
x_1^2 & 0 & - x_4 \\
- x_4^2 & x_2^2 & 0  \\
0 & - x_4 & x_3^4 
\end{array}\right)
$$
and $\mathbf{m} = (x_1, x_2^2, x_3).$ In this case, 
$$
M = 
\left(\begin{array}{rrrr}
     3 &     0 &    -1 &    -1\\
    -1 &    4 &   0 &    -2\\
     0 &   -2 &     5 &    -1\\
    -2 &    -2 &    -4 &     4
\end{array}\right).
$$
By computing the Smith normal form of $M$ (using GAP, for instance), we may take $\mathcal{A} = \{(4,14), (5,18), (4,13), (8,29)\} \subset \mathbb{Z}/2\mathbb{Z} \oplus \mathbb{Z}$. 

Since $\mathbb{Z}^4/\mathcal{M} \cong \mathbb{Z} \oplus \mathbb{Z}/2\mathbb{Z}$ the corresponding critical ideal of Northcott type 
$$
I_\mathcal{A} = \langle x_1^3-x_3x_4,x_2^4-x_1x_4^2,x_3^5-x_2^2x_4, x_1^2x_2^2x_3^4-x_4^4
\rangle \subset \Bbbk[x_1, \ldots, x_n],
$$ 
is not prime by Proposition \ref{Prop PrimeSat}. Moreover, by Theorem \ref{Corolario 2.2}, we have that that $I_\mathcal{A}$ is radical and has two associated primes if $\mathrm{car}(\Bbbk) \neq 2$ and $J$ is primary if $\mathrm{car}(\Bbbk) = 2$. 

If $\mathrm{car}(\Bbbk) \neq 2$ the associated primes are minimally generated by
$$
\{ x_1^3-x_3x_4,x_2^4-x_1x_4^2,x_3^5-x_2^2x_4, x_1x_2x_3^2 \pm x_4^2, x_2x_3^3 \pm x_1^2x_4 
\}.
$$
Clearly, they are not critical ideals.
\end{example}

\section{Numerical semigroups of Northcott type}\mbox{}\par

In this section we focus on monoids of Northcott type that are numerical semigroups. Up to isomorphism, this means that the torsion part is equal to zero. We will study the invariants of the semigroup in terms of the corresponding exponents in the Northcott type ideal.

Let us see that it is possible to construct a subset $\mathcal{A}$ of $\mathbb{N}$ such that $\mathbb N \mathcal A$ is a numerical semigroup of Northcott type. To this end, we consider the critical Northcott type ideal $J$ generated by $\{f_1, \ldots, f_{n-1}, D\}$ as in Definition \ref{Def CINT}
for $\mathbf m = (x_1, \ldots, x_{n-1})$. 

Let $M$ be the corresponding matrix to the exponent vectors of the given generators of $J$, that is to say, \[ M = 
\begin{pmatrix}
u_{n1}+1 & 0 & 0 & \ldots & 0 & -1 & -u_{1n} \\
- 1 & u_{n2}+1 & 0 & \ldots & 0 & 0 & -u_{2n} \\
0 & -1 & u_{n3}+1 & \ldots & 0 & 0 & -u_{3n} \\
\vdots & \vdots & \vdots & & \vdots & \vdots & \vdots \\
0 & 0 & 0 & \ldots & -1 & u_{n\, n-1}+1 & -u_{n-1\, n} \\ 
-u_{n1} & -u_{n2} & -u_{n3} & \ldots & -u_{n\, n-2} & -u_{n\, n-1}  & \sum_{i=1}^{n-1} u_{in}   
\end{pmatrix}.
\] 
By Theorem \ref{th:gen-N} there exist a commutative finite group $T$ and $\mathcal{A} \subset T \oplus \mathbb{N}$ such that $I_\mathcal{A} = J$. Observe that a necessary and sufficient condition for the set $\mathcal{A}$ to be a system of generators of a numerical semigroup is that $T=0$, equivalently, that the nonzero $d_i$'s appearing in the proof of Theorem \ref{th:gen-N} are equal to $1$. Recall that the $d_i$ are nothing but the invariant factors of $M$. Thus they are equal to greatest common divisor of the $i \times i$ minors of the matrix $M$ in Notation \ref{not:gen-N}, for each $i \in \{1, \ldots, n\}$, and $d_i$ divides $d_{i+1}$ for every $i \in \{1, \ldots, n-1\}$. 

\begin{proposition}\label{prop:gen-nsgps1}
With the above notation, we have $d_1 = \dots = d_{n-2} = 1$ and $d_{n-1}$ is equal to \begin{equation}\label{ecu1} \gcd\left(\prod_{i=1}^{n-1} (u_{ni}+1) - 1 ,u_{1n} + \sum_{i=1}^{n-2}\left(\prod_{j=1}^i (u_{nj}+1) \right) u_{i+1\, n} \right).\end{equation}
\end{proposition}

\begin{proof}
Since the $(n-2) \times (n-2)$ minor of $M$ obtained by removing the first and the last rows and the two last columns is equal to $(-1)^{n-2}$, we have that $d_{n-2} = 1$ and therefore $d_1 = \ldots = d_{n-3} = 1$. We already know that $d_n = 0$, because the last row of $M$ is the sum of the previous ones. So, it suffices to study the $(n-1) \times (n-1)$ minors of the first $n-1$ rows of $M$. Let $a_i$ be absolute value of the $(n-1)\times(n-1)$ minor of $M$ obtained by removing the last row and the $i$th column of $M$. By direct computation, we have that 
\begin{itemize}
\item $a_n=\prod_{i=1}^{n-1} (u_{ni}+1)-1$, 
\item $a_k=u_{\sigma(k)n}+\big(\sum_{i=1}^{n-2}\prod_{j=1}^i (u_{n\, \sigma^{k}(j)} + 1)\big) u_{\sigma^k(i+1)\, n}$, for $k\in\{1,\ldots, n-1\}$, 
\end{itemize}
where $\sigma$ is the permutation defined by the cycle $(1 \cdots n-1)$. Now, since 
\begin{equation}\label{ecu2}(u_{nk} + 1) a_k - a_{\sigma^{-1}(k)} - u_{k n} a_n  = 0,
\end{equation} 
for all $k \in\{1, \ldots, n-1\}$, we easily derive $\gcd(a_1, \ldots, a_{n-1}, a_n) = \gcd(a_{n-2}, a_n)$. 
\end{proof}

\begin{example}
By the above proposition, if we take $\mathbf{m} = (x_1, \ldots, x_{n-1}),\ u_{nj} = 1$, $j \in\{ 1, \ldots, n-1\}$, $u_{in} = 1$, $i\in\{ 2, \ldots, n-1\}$ and $u_{1n} = 2$, then the subsemigroup of $\mathbb{N}$ generated by $$\mathcal{A} = \left\{2^{n-1}+\sum_{k=0}^{n-3} 2^k, 2^{n-1}+\sum_{k=0}^{n-4} 2^k, \ldots,2^{n-1}+1,2^{n-1},2^{n-1}-1\right\}$$ is a numerical semigroup of Northcott type.
\end{example}

Once it is shown that numerical semigroups of Northcott type exist, let us explicitly compute their fundamental invariants.

Let $S$ be a numerical semigroup, and let $z$ be an integer. The \textbf{Ap\'ery set} of $z$ is the set 
\[
\mathrm{Ap}(S,z)=\{s\in S\mid s-z\not\in S\}.
\]
It is well known that the cardinality of $\mathrm{Ap}(S,z)$ is greater than or equal to $z$ and that the equality holds if and only if $z \in S$ (see \cite[Lemma 2.4]{libro-ns}).

\begin{lemma}\label{lem:Apery}
If $\mathbb N \mathcal A$ is a numerical semigroup of Northcott type, then 
\[
\mathrm{Ap}(\mathbb N \mathcal A,a_n)=\left\{ \sum_{i=1}^{n-1}u_ia_i  ~ \middle|~ \begin{matrix}0 \leq u_i < u_{ni}+u_{i+1\, i},\\  i\in\{ 1, \ldots, n-2\},\\ 0 \leq u_{n-1} < u_{n\, n-1}+u_{1\, n-1}\, \\ \text{and}\ u_j < u_{n\,j}\ \text{for some}\ j \end{matrix}\right \}.
\]
\end{lemma}

\begin{proof}
By the proof of Theorem \ref{th:gen-N}, we have that $\mathcal{G} = \{f_1, \ldots, f_{n-1}, D\}$ is a Gr\"obner basis of $I_\mathcal{A}$ with respect to a monomial order on $\Bbbk[\mathbf{x}]$ represented by a matrix whose first row is $(a_1, \ldots, a_n)$ and second row is $(0, \ldots, 0, -1)$. Now, our claim follows from the results in \cite[Section 3]{MOT} or more directly by \cite[Theorem 3]{OjVi3}, by just taking into account that the initial ideal of $I_\mathcal{A}$, with respect to an $\mathcal{A}-$graded reverse lexicographical term order such that $x_n$ is the least variable, is spanned by $\{ x_1^{u_{n1}+u_{21}}, \ldots, x_{n-2}^{u_{n\, n-2}+u_{n-1\, n-2}}, x_{n-1}^{u_{n\, n-1}+u_{1\, n-1}}, x_1^{u_{n1}} \cdots x_{n-1}^{u_{n\, n-1}} \}$, as we have shown in the proof of Theorem \ref{th:gen-N}.
\end{proof}

The \textbf{Frobenius number} of $S$ is the largest integer not in $S$, denoted by $\mathrm F(S)$. The \textbf{genus} of $S$, $\mathrm g(S)$, is the number of gaps of $S$, that is, the cardinality of $\mathbb N\setminus S$. These two invariants can be calculated from $\mathrm{Ap}(S,s)$ for $s\in S\setminus \{0\}$, by using Selmer's formulas (see for instance \cite[Proposition 2.12]{libro-ns}):
\begin{enumerate}
\item $\mathrm F(S)=\max(\mathrm{Ap}(S,s))-s$,
\item $\mathrm g(S)=\frac{1}s\sum_{w\in\mathrm{Ap}(S,s)} w - \frac{s-1}2$.
\end{enumerate}

We say that an integer $z$ is a pseudo-Frobenius number if $z \not\in S$ and $z + s \in S$ for all $s \in S \setminus \{0\}$. We will denote by $\mathrm{PF}(S)$ the set of \textbf{pseudo-Frobenius numbers} of $S$, and its cardinality, is the \textbf{type} of $S$, denoted by $\mathrm t(S)$. It is known that the type of a numerical semigroup agrees with the Cohen-Macaulay type of $\Bbbk[\![S]\!]$.

\begin{corollary}\label{lem:pseudoFrob}
If $\mathbb N \mathcal A$ is a numerical semigroup of Northcott type, $\sigma$ is the permutation determined by the cycle $(1\, 2\, \cdots\, n-1)$ and $N = \Big(\sum_{i=1}^{n-2} (u_{ni}+u_{i+1\,i}-1) a_i\Big) + (u_{n\, n-1}+u_{1\, n-1}-1) a_{n-1}$, then 
\[\mathrm{PF}(\mathbb N \mathcal A) = \left\{N - u_{\sigma(j)\,j} a_j-a_n \mid j\in\{1,\ldots ,n-1\} \right\}.\]
In particular, $\mathrm t(\mathbb N \mathcal A) = n-1$, and $\mathrm{F}(\mathbb N \mathcal A) = N - \min\{ u_{\sigma(j)\, j} a_j \mid j \in\{ 1, \ldots, n\}\} -a_n$. 
Moreover,
\begin{align*}
\mathrm{g}(\mathbb N \mathcal A) = &  \sum_{i=1}^{n-1} \frac{(u_{ni}+u_{\sigma(i)\, i}-1)(u_{ni}+u_{\sigma(i)\, i})}{2} \prod_{j \neq i} (u_{nj}+u_{\sigma(j)\, j}) \frac{a_i}{a_n}  \\ & - \sum_{i=1}^{n-1} \left(\frac{(u_{\sigma(i)\,i}-1)u_{\sigma(i)\,i}}{2} + u_{ni} u_{\sigma(i)\, i} \right)\prod_{j \neq i} u_{\sigma(j)\,j} \frac{a_i}{a_n} - \frac{a_n - 1}{2}
\end{align*}
\end{corollary}

\begin{proof}
By \cite[Proposition 2.20]{libro-ns}, we have that $$\mathrm{PF}(\mathbb N \mathcal A) = \{w - a_n \mid w \in \mathrm{Maximals}_{\leq_\mathcal{A}} \mathrm{Ap}(\mathbb N \mathcal A, a_n) \}.$$ Now, by Lemma \ref{lem:Apery}, our claim on the set of pseudo-Frobenius numbers follows. Notice that $N - u_{\sigma(j)\,j} a_j-a_n \neq N - u_{\sigma(k)\,k} a_k-a_n$ if $k \neq j$; 
otherwise $u_{\sigma(j)\,j} a_j = u_{\sigma(k)\,k} a_k$ in contradiction with the minimality of $c_j = u_{nj}+u_{\sigma(j)\, j}$. 

From the definition of pseudo-Frobenius number we get that $\mathrm{F}(S) = \max(\mathrm{PF}(S))$, for any numerical semigroup $S$. Finally, the last formula is a consequence of Selmer's formula for the genus of a numerical semigroup.
\end{proof}

Let $S$ be a numerical semigroup minimally generated by $n$ elements. Then we say that $S$ has \textbf{embedding dimension} equal to $n$, and we denote it as $\operatorname{e}(S)=n$. Other invariant associated to $S$ is $\operatorname{n}(S)$ that denotes the number of elements in $S$ smaller than $\operatorname{F}(S)$. In 1978 H. S. Wilf in \cite{wilf} conjectured that for any numerical semigroup $S$ the relation $\operatorname{F}(S)+1\leq \operatorname{e}(S)\operatorname{n}(S)$ holds true. This conjecture remains open. Next corollary shows that Wilf's conjecture holds for numerical semigroups of Northcott type.

\begin{corollary}
Assume that $\mathbb{N}\mathcal{A}$ is a numerical semigroup of Northcott type. Then $\mathbb{N}\mathcal{A}$ fulfills Wilf's conjecture. 
\end{corollary}
\begin{proof}
From \cite[Theorem 20]{fgh} we have that $\operatorname{F}(\mathbb{N}\mathcal{A})+1\le (\operatorname{t}(\mathbb{N}\mathcal{A})+1)\operatorname{n}(\mathbb{N}\mathcal{A})$. If $\mathbb{N}\mathcal{A}$ is a numerical semigroup of Northcott type, by Corollary \ref{lem:pseudoFrob}, $\operatorname{t}(\mathbb{N}\mathcal{A})=\operatorname{e}(\mathbb{N}\mathcal{A})-1$. Thus, in this setting, we get $\operatorname{F}(\mathbb{N}\mathcal{A})+1\le \operatorname{e}(\mathbb{N}\mathcal{A})\operatorname{n}(\mathbb{N}\mathcal{A})$. This means that $\mathbb{N}\mathcal{A}$ fulfills Wilf's conjecture.
\end{proof}

Now, we are going to focus in the case $\mathbf m = (x_1, \ldots, x_{n-1})$, in order to simplify the formulae in Corollary \ref{lem:pseudoFrob} as follows: 
\[\mathrm{PF}(\mathbb N\mathcal A)=\left\{\left(-1+\sum\nolimits_{i=1}^{n-1}u_{in}\right)a_n-a_j\mid j\in\{1,\ldots ,n-1\}\right\}\] 
and $$\mathrm{g}(\mathbb N \mathcal A)=\frac{a_n-1}{2}\left(-1+\sum\nolimits_{i=1}^{n-1}u_{in}\right).$$ The first formula is a consequence of the equalities $N=\sum_{i=1}^{n-1}u_{ni}a_i =\big(\sum_{i=1}^{n-1}u_{in}\big) a_n$ (notice that $\sum_{i=1}^{n-1}u_{in}$ is the last entry of the last row in the matrix $M$). The proof of the formula for the genus requires some extra computations:
\begin{align*}
\mathrm{g}(\mathbb N \mathcal A) = &  \sum_{i=1}^{n-1} \frac{u_{ni}(u_{ni}+1)}{2} \prod_{j \neq i} (u_{nj}+1) \frac{a_i}{a_n}- \sum_{i=1}^{n-1} u_{ni} \frac{a_i}{a_n} - \frac{a_n - 1}{2} \\
= &  \sum_{i=1}^{n-1} \frac{u_{ni}a_i}{2a_n} \prod_{j=1}^{n-1}(u_{nj}+1)-  \frac{\sum_{i=1}^{n-1} u_{ni}a_i}{a_n} - \frac{a_n - 1}{2} \\
= &  \prod_{j=1}^{n-1}(u_{nj}+1)\frac{\sum_{i=1}^{n-1} u_{in}}{2}-\sum_{i=1}^{n-1} u_{in}-\frac{a_n - 1}{2} \\
= &  (a_n+1)\frac{\sum_{i=1}^{n-1} u_{in}}{2}-\sum_{i=1}^{n-1} u_{in}-\frac{a_n - 1}{2} \\
= &  \left(\frac{a_n+1}{2}-1\right)\sum_{i=1}^{n-1} u_{in}-\frac{a_n - 1}{2} =  \frac{a_n-1}{2}\left(-1+\sum_{i=1}^{n-1}u_{in}\right).
\end{align*}

\begin{example}
Let $n = 5$. With the same notation as in Proposition \ref{prop:gen-nsgps1}, if $u_{5i} = i$, $i\in\{ 1, \ldots, 4\}$,  $u_{15} = 5$ and $u_{i5} = i$, $i \in\{ 2, \ldots, 4\}$, then \[I_\mathcal{A} = \langle x_3^4-x_2x_5^3, x_2^3-x_1x_5^2, x_4^5-x_3x_5^4, x_1^2-x_4x_5^5, x_1x_2^2x_3^3x_4^4-x_5^{14} \rangle \subset \Bbbk[x_1,\ldots, x_5]\] and $\mathcal{A} = \{359,199,139,123,119\}$. It is easy, from the last remark, to obtain
\begin{align*}
\mathrm{PF}(\mathbb N\mathcal A&=\{119\cdot 13-359,119\cdot 13-199,119\cdot 13-139,119\cdot 13-123\}\\ &=\{1188, 1348, 1408, 1424\}
\end{align*} 
and $\mathrm{g}(\mathbb N \mathcal A)=(14-1)(119-1)/2=767$.
\end{example}

\section{Nonunique factorization invariants for Northcott type semigroups}\label{sec:fact}

Some nonunique factorization invariants can be derived from a minimal presentation of the monoid (see for instance \cite{fact-review}). The maximum and minimum of the Delta set of the monoid and the catenary degree can be obtained once we know the elements in the monoid involved in a minimal presentation.

Let $G$ be a finitely generated commutative group, and let $\mathcal{A} = \{\mathbf{a}_1, \ldots, \mathbf a_n\} \subset G$ be such that $\mathbb{N}\mathcal{A}$ is unit free. Recall that an element $\mathbf{b}\in \mathbb{N}\mathcal{A}$ is a \textbf{Betti element} of $\mathbb{N}\mathcal{A}$ if there exists $\mathbf{x}^{\mathbf{u}}-\mathbf{x}^\mathbf{v}$ with ${A}\mathbf{u}=\mathbf{b}$($=A\mathbf{v}$), where as usual $A$ denotes the matrix with columns the elements of $\mathcal{A}$. 

For $\mathbf{s}\in \mathbb{N}\mathcal{A}$, the \textbf{set of factorizations} of $\mathbf{s}$ is the set 
\[
\mathsf{Z}(\mathbf{s})= \{\mathbf{u}\in\mathbb{N}^n \mid A\mathbf{u}=\mathbf{s}\}.
\]
As $\mathbb{N}\mathcal{A}$ is unit free, the set $\mathsf{Z}(\mathbf{s})$ has finitely many elements (this follows by Dickson's Lemma, since they are all incomparable with respect to the usual partial order on $\mathbb{N}^n$). The monoid $\mathbb{N}\mathcal{A}$ is a \textbf{unique factorization monoid} if every element has a unique factorization. Notice that a semigroup of Northcott type is not a unique factorization monoid, since its Betti elements have at least two factorizations.

The \textbf{length} of a factorization $\mathbf{u}=(u_1,\ldots,u_n)$ of $\mathbf{s}$ is defined as $|\mathbf{u}|=u_1+\cdots+u_n$. The set of lengths of factorizations of $\mathbf{s}$ is thus defined as 
\[
\mathsf{L}(\mathbf{s})=\{ \lvert\mathbf{u}\rvert \mid \mathbf{u}\in \mathsf{Z}(\mathbf{s})\}.
\]
The monoid $\mathbb{N}\mathcal{A}$ is said to be \textbf{half-factorial} if for all $\mathbf s\in\mathbb{N}\mathcal{A}$ the set $\mathsf{L}(\mathbf{s})$ is a singleton. A way to measure how far a monoid is from being half-factorial is by means of the distances between consecutive factorization lengths. This is precisely the idea under the definition of Delta sets.

Since $\mathsf{L}(\mathbf{s})$ has finitely many elements, we can arrange its elements and write it as $\mathsf{L}(\mathbf{s})=\{l_1<\cdots<l_t\}$. The \textbf{Delta set} of $\mathbf{s}$ is then defined as 
\[
\Delta(\mathbf{s})=\{ l_2-l_1,\ldots, l_t-l_{t-1}\}.
\]
The \textbf{Delta set} of $\mathbb{N}\mathcal{A}$ is 
\[
\Delta(\mathbb{N}\mathcal{A}) = \bigcup_{\mathbf{s}\in \mathbb{N}\mathcal{A}}\Delta(\mathbf{s}).
\]

\begin{proposition}\label{prop:delta-nt}
Let $G$ be a finitely generated commutative group, and let $\mathcal{A} = \{\mathbf{a}_1, \ldots, \mathbf a_n\} \subset G$ be such that $\mathbb{N}\mathcal{A}$ is of Northcott type. Let $u_{ij}$ be as in \eqref{ecPhi}. Then 
\begin{align*}
\min(\Delta(\mathbb{N}\mathcal{A})) = \gcd(\{&|u_{n1}+u_{21}-(u_{1\,n-1}+u_{i1})|, |u_{n2}+u_{32}-(u_{21}+u_{2n})|,
\\ &\ldots, |u_{n\,n-1}+u_{1\, n-1}-(u_{n-1\,n-2}+u_{n-1\,n})|,
\\ &|\sum\nolimits_{i=1}^{n-1} u_{in}-\sum\nolimits_{i=1}^{n-1} u_{ni}|\}),
\end{align*}
and 
\begin{align*}
\max(\Delta(\mathbb{N}\mathcal{A})) =
\max(\{&|u_{n1}+u_{21}-(u_{1\,n-1}+u_{i1})|, |u_{n2}+u_{32}-(u_{21}+u_{2n})|,
\\ &\ldots, |u_{n\,n-1}+u_{1\, n-1}-(u_{n-1\,n-2}+u_{n-1\,n})|,
\\ &|\sum\nolimits_{i=1}^{n-1} u_{in}-\sum\nolimits_{i=1}^{n-1} u_{ni}|\}).
\end{align*}
\end{proposition}
\begin{proof}
By Corollary \ref{cor:uniq-pres}, $\mathbb{N}\mathcal{A}$ is uniquely presented. In particular, this means that every Betti element of $\mathbb{N}\mathcal{A}$ has exactly two different factorizations, \cite{GSOj}. These factorizations are precisely the exponents of the binomials $f_1,\ldots, f_{n-1},D$. Hence for each Betti element its Delta set is just the length of the difference of the corresponding exponents. With this observation, the proof now follows from \cite[Corollary 2.4, Theorem 2.5]{delta-bf}.
\end{proof}

Now we see that the catenary degree of a semigroup of Northcott type can also be determined from the $u_{ij}$ in \eqref{ecPhi}. The catenary degree measures how spread are the factorizations of the elements of a monoid. 

Given two elements $\mathbf{u},\mathbf{v}\in \mathbb{N}^n$, define $\mathbf{u}\wedge\mathbf{v}$ as the minimum componentwise of $\mathbf{u}$ and $\mathbf{v}$, that is, the infimum in the lattice $\mathbb{N}^n$ endowed with the usual partial ordering. The \textbf{distance} between $\mathbf{u}$ and $\mathbf{v}$ is 
\[
\mathrm{d}(\mathbf{u},\mathbf{v})=\max\{ |\mathbf{u}-(\mathbf{u}\wedge \mathbf{v})|, |\mathbf{v}-(\mathbf{u}\wedge \mathbf{v})|\}.
\]
Let $N$ be a nonnegative integer, and let $\mathbf{s}\in \mathbb{N}\mathcal{A}$. For $\mathbf{u},\mathbf{v}\in\mathsf{Z}(\mathbf{s})$, an \textbf{$N$-chain} joining $\mathbf{u}$ and $\mathbf{v}$ is a sequence $\mathbf{u}^1,\ldots, \mathbf{u}^k\in\mathsf{Z}(\mathbf{s})$ such that $\mathbf{u}=\mathbf{u}^1$, $\mathbf{v}=\mathbf{u}^k$, and $\mathrm{d}(\mathbf{u}^i,\mathbf{u}^{i+1})\le N$ for all $i\in \{1,\ldots, k-1\}$. 

The \textbf{catenary degree} of $\mathbf{s}$, $\mathsf{c}(\mathbf{s})$, is the least positive integer $c$ such that for every two factorizations of $\mathbf{s}$, there exists a $c$-chain joining them.  The \textbf{catenary degree} of $\mathbb{N}\mathcal{A}$ is the supremum of the catenary degrees of its elements. 

\begin{proposition}\label{prop:catenary-nt}
Let $G$ be a finitely generated commutative group, and let $\mathcal{A} = \{\mathbf{a}_1, \ldots, \mathbf a_n\} \subset G$ be such that $\mathbb{N}\mathcal{A}$ is of Northcott type. Let $u_{ij}$ be as in \eqref{ecPhi}. Then 
\begin{align*}
\mathsf{c}(\mathbb{N}\mathcal{A})) = \max\big(\big\{&\max\{u_{n1}+u_{21},u_{1\,n-1}+u_{i1}\}, \max\{u_{n2}+u_{32},u_{21}+u_{2n}\},
\\ &\ldots, \max\{u_{n\,n-1}+u_{1\, n-1}, u_{n-1\,n-2}+u_{n-1\,n}\},
\\ &\max\{\sum\nolimits_{i=1}^{n-1} u_{in},\sum\nolimits_{i=1}^{n-1} u_{ni}\}\big\}\big).
\end{align*}
\end{proposition}
\begin{proof}
The proof goes as in Proposition \ref{prop:delta-nt}. Since every Betti element has exactly two factorizations, its catenary degree is the maximum of their lengths. The proof follows by \cite[Theorem 3.1]{c-t}.
\end{proof}

\section{Gluings and Northcott type ideals}\label{sect:gl-nt}

In this section, we give a Northcott type description of the toric ideals generated by (gN1) and (gN2), and we give a method to produce critical binomial ideals for arbitrary $n$ that are not Northcott type in general.

Let $G$ be a finitely generated commutative group and let $\mathcal{A} = \{\mathbf{a}_1, \ldots, \mathbf a_n\} \subset G$ be such that $\mathbb N \mathcal A \cap (- \mathbb N \mathcal A) = \{0\}$. Assume that $\mathcal A=\mathcal A_1\cup\mathcal A_2$ with $\mathcal A_1\neq \varnothing$ and  $\mathcal A_2 \neq \varnothing$. After reindexing the elements in $\mathcal A$ if necessary, we may suppose that $\mathcal A_1=\{\mathbf a_1,\ldots, \mathbf  a_k\}$ and  $\mathcal A_2=\{\mathbf a_{k+1},\ldots, \mathbf a_n\}$. We say that $\mathbb N\mathcal A$ is a \textbf{gluing} of $\mathbb N\mathcal A_1$ and $\mathbb N\mathcal A_2$ if $I_\mathcal A$ has a system of generators of the form $B_1\cup B_2\cup \{\mathbf x^\mathbf u-\mathbf x^\mathbf v\}$, where $B_1\cup\{\mathbf x^\mathbf u\}\subseteq \Bbbk[x_1,\ldots, x_k]$ and $B_2\cup \{\mathbf x^\mathbf v\}\subseteq \Bbbk[x_{k+1},\ldots, x_n]$. The concept gluing of semigroups was defined by Rosales in \cite{Ros97} for affine semigroups, but the results obtained in the first section of this paper can be applied in our context (see for instance \cite{acc}). We will only consider here proper gluings, that is, $\mathcal A$ will be a minimal generating system of $\mathbb N\mathcal A$.

It follows from the proof of \cite[Theorem 1.4]{Ros97} that if $\mathbb N\mathcal A$ is the gluing of $\mathbb N \mathcal A_1$ and $\mathbb N\mathcal A_2$, and $\mathbf u$ and $\mathbf v$ are as in the preceding paragraph, then $\mathbb Z\mathcal A_1\cap \mathbb Z\mathcal A_2= d\mathbb Z$, where $d=\mathcal A \mathbf u=\mathcal A\mathbf v$. 

If $\mathbb N\mathcal A$ is a numerical semigroup minimally generated by $\mathcal A$, then without loss of generality we may assume that $\mathcal A \subset \mathbb N$. Now, if $d_i=\gcd\mathcal A_i$, $i\in \{1,2\}$, then $\mathbb N\mathcal A$ is a gluing of $\mathbb N\mathcal A_1$ and $\mathbb N\mathcal A_2$ if and only if $d_1\in \mathbb N\mathcal A_2\setminus \mathcal A_2,\ d_2\in\mathbb N\mathcal A_1\setminus \mathcal A_1$ and $\gcd(d_1,d_2)=1$ (see \cite[Chapter 8]{libro-ns}).

Before the next result, we observe that we can consider (gN1) as a subcase of (gN2) by taking $u_{12} = u_{32} = 0$ and, consequently, $f_1 = D$. Let us suppose now that $J$ is an ideal of $\Bbbk[\mathbf{x}]$ generated by (gN2) extended.

The following result is the counterpart of Proposition \ref{Prop Th5Bresinsky} and its proof follows from \cite[Propositions 3.3 and 3.4]{Villarreal}.

\begin{proposition}\label{prop:gN*}
Let $B$ be the matrix whose rows are the exponent vectors of $f_1, f_2$ $D$, and $g$, that is,
$$
B = \left(\begin{array}{cccc}
c_1 &  -u_{12} & -u_{13} & 0 \\
-u_{21} & c_2 & - u_{23} & 0 \\
-u_{31} & -u_{32} & c_3 & 0 \\
-u_{41} & -u_{42} &  -u_{43} & c_4
\end{array}\right).
$$
Notice that the first and the third rows might define the same binomial.
Then there exists $\mathcal{A} \subset \mathbb{N}$ such that $J = I_\mathcal{A}$ if and only if 
\begin{itemize}
\item[(a)]  $\{f_1, f_2, D\}$ generates a toric ideal in $\Bbbk[x_1, x_2, x_3]$.
\item[(b)] the $(3,j)-$minors of $B$, $j \in\{ 1, \ldots, 4\}$ are relatively prime.
\end{itemize}
\end{proposition}

\begin{corollary}
With the above notation, there exists $\mathcal{A} = \{a_1, a_2, a_3, a_4\} \subset \mathbb N$ such that $J = I_\mathcal{A}$ if and only if $\mathbb N \mathcal A$ is a gluing of $c_4 \mathbb N\mathcal{A}'$ and $a_4 \mathbb{N}$, with $\mathcal{A}' = \{  a_i/\mathrm{gcd}(a_1, a_2, a_3) \mid i = 1,2,3\}$. 
\end{corollary}

\begin{proof}
By \cite[Proposition 4.13(a)]{Sturmfels95}, $I_\mathcal{A} \cap \Bbbk[x_1, x_2, x_3]$ is the toric ideal of $\mathcal{A}'$. So, condition (a) in Proposition \ref{prop:gN*} implies that $I_\mathcal{A} = I_{\mathcal{A}'} + \langle g \rangle$. Furthermore, by Proposition \ref{prop:gN*}(b), $a_i = c_4 a'_i$, $i \in\{ 1, 2, 3\}$, and $c_4 a_4 = u_{41} a_1 + u_{42} a_2 + u_{43} a_3 = c_4 (u_{41} a'_1 + u_{42} a'_2 + u_{43} a'_3),$ so we conclude that $\mathcal{A}$ is a gluing of $c_4 \mathcal{A}' \mathbb N$ and $a_4 \mathbb{N}$. The converse is a direct consequence of Proposition \ref{prop:gN*} and the Herzog's classification theorem of toric ideals associated to numerical semigroups minimally generated by three positive integers (see \cite[Section 3]{Herzog70}, or \cite[Theorem 2.2]{OjPis}).
\end{proof}

Let us see now that the gluing of a critical monoid with a free semigroup of rank one is again a critical monoid.

\begin{proposition}\label{prop: gluing-critical}
Let $G$ be a commutative group, $\mathcal A = \{\mathbf{a}_1, \ldots, \mathbf{a}_n\} \subset G $ and $\mathbf{a}_{n+1} \in G$. 
If $\mathbb N \mathcal{A}$ is a critical monoid and $S=\mathbb N\mathcal A+\mathbb N\mathbf a_{n+1}$ is a gluing of $\mathbb N \mathcal{A}$ and $\mathbb N \mathbf{a}_{n+1}$, then $S$ is a critical monoid.
\end{proposition}

\begin{proof}
We are assuming that $I_\mathcal A$ is critical, and so it is  generated by a set of the form  $\{x_1^{c_1}-\mathbf{x}^{\mathbf{u}_1},\ldots, x_n^{c_n}-\mathbf{x}^{\mathbf{u}_n}\} \subset \Bbbk[x_1, \ldots, x_n]$, for some $\mathbf{u}_1, \ldots, \mathbf{u}_n \in\mathbb N^n$, and where $c_i$ is the least multiple of $\mathbf a_i$ that belongs to $\mathbb N(\mathcal A\setminus\{\mathbf a_i\})$ (and consequently the $i^{th}$ coordinate of $\mathbf{u}_i$ is zero). On the other hand, by definition of gluing, the semigroup ideal of $S$, $I_{\mathcal A\cup\{\mathbf a_{n+1}\}}$,  has a system of generators of the form $\mathcal{S} = \{x_1^{c_1}-\mathbf{x}^{\mathbf{u}_1},\ldots, x_n^{c_n}-\mathbf{x}^{\mathbf{u}_n}\} \cup \{x_{n+1}^{c_{n+1}} - \mathbf{x}^{\mathbf{u}_{n+1}}\}$, with $\mathbf{x}^{\mathbf{u}_{n+1}} \in \Bbbk[x_1, \ldots, x_n]$ (recall that $I_{\mathbb N \mathbf{a}_{n+1}} = \langle 0 \rangle$). Also, we know that  $\mathbb Z \mathcal A \cap \mathbb Z \mathbf{a}_{n+1} = \mathbb Z (c_{n+1} \mathbf a_{n+1})$ (our definition of gluing implies that $S$ is minimally generated by $\mathcal A\cup\{\mathbf a_{n+1}\}$, and thus $c_{n+1}>1$). 

So it only remains to prove that the $c_i$ are still the minimal positive integers such that $c_i\mathbf a_i\in \mathbb N((\mathcal A\cup\{\mathbf a_{n+1}\})\setminus\{\mathbf a_i\})$, and that $c_{n+1}$ is the minimum positive integer with $c_{n+1}\mathbf a_{n+1}\in \mathbb N\mathcal A$.

Let $i\in \{1,\ldots,n\}$ and $k$ a positive integer fulfilling that $k\mathbf a_i = \sum_{j=1}^n w_j \mathbf a_j +w_{n+1}\mathbf a_{n+1}$ for some $\mathbf w = (w_1, \ldots, w_n, w_{n+1}) \in\mathbb N^{n+1}$ with $w_i=0$. Then $w_{n+1}\mathbf a_{n+1}\in \mathbb Z\mathcal A\cap \mathbb Z\mathbf a_{n+1}=\mathbb Z(c_{n+1}\mathbf a_{n+1})$. It follows that $c_{n+1}\mid w_{n+1}$, thus we may write $w_{n+1}= \ell c_{n+1}$. As $c_{n+1}\mathbf a_{n+1}= \sum_{i=1}^n u_{n+1\, i} \mathbf a_i$, we have that $(k-\ell u_{n+1\, i})\mathbf a_i = (w_1+\ell u_{n+1\, 1})\mathbf a_1+\cdots + (w_{i-1}+\ell u_{n+1\, i-1})\mathbf a_{i-1}+(w_{i+1}+\ell u_{n+1\, i+1})\mathbf a_{i+1}+\cdots + (w_{n}+\ell u_{n+1\, n})\mathbf a_{n}$. Since $\mathbb N\mathcal A$ is free of units, we deduce that $k-\ell u_{n+1\, i}>0$. The minimality of $c_i$ implies that $c_i\le k-\ell u_{n+1\, i}\le k$.

Finally, if $k$ is a positive integer such that $k\mathbf a_{n+1}\in \mathbb N\mathcal A$, then $k\mathbf a_{n+1}\in \mathbb Z\mathcal A\cap \mathbb Z\mathbf a_{n+1}=\mathbb Z(c_{n+1}\mathbf a_{n+1})$. This forces $c_{n+1}\mid k$, whence $c_{n+1}\le k$.
\end{proof}

Observe also that we can use Theorem \ref{th:gen-N} and Proposition \ref{prop: gluing-critical} to produce a new critical monoid. In fact, we can repeat this process and gluing again with another copy of $\mathbb N$ and so on. However, it is convenient to point that, for $n > 4$, there are numerical semigroups whose critical associated ideal is not of this form nor Northcott type.

\begin{example}
A quick search using \texttt{numericalsgps} (\cite{numericalsgps}, a \texttt{GAP} package, \cite{GAP}), shows that $S=\left< 11, 13, 14, 15, 19 \right>$ is not a gluing and not Northcott type. 
	\begin{verbatim}
	gap> MinimalPresentationOfNumericalSemigroup(s);
     [ [ [ 0, 0, 0, 2, 0 ], [ 1, 0, 0, 0, 1 ] ],
       [ [ 0, 0, 2, 0, 0 ], [ 0, 1, 0, 1, 0 ] ],
       [ [ 0, 2, 0, 0, 0 ], [ 1, 0, 0, 1, 0 ] ],
       [ [ 1, 1, 1, 0, 0 ], [ 0, 0, 0, 0, 2 ] ],
       [ [ 3, 0, 0, 0, 0 ], [ 0, 0, 1, 0, 1 ] ] ]
	\end{verbatim}
\end{example}

Unfortunately in the process of applying Proposition \ref{prop: gluing-critical}, the uniquely presented property of the family (N) might be lost. This is due to the fact that $c_{n+1} \mathbf{a}_{n+1}$ can be chosen to have more than one factorization in terms of the elements of $\mathcal A$. We illustrate this with an example.

\begin{example}
Let $S=\langle 3,5,7\rangle$ and take $a=2$ and $b=21$. 
\begin{verbatim}
gap> s:=NumericalSemigroup(3,5,7);
<Numerical semigroup with 3 generators>
gap> FactorizationsElementWRTNumericalSemigroup(21,s);
[ [ 7, 0, 0 ], [ 2, 3, 0 ], [ 3, 1, 1 ], [ 0, 0, 3 ] ]
gap> t:=NumericalSemigroup(6,10,14,21);
<Numerical semigroup with 4 generators>
gap> AsGluingOfNumericalSemigroups(t);
[ [ [ 6, 10, 14 ], [ 21 ] ] ]
gap> IsUniquelyPresentedNumericalSemigroup(t);
false
\end{verbatim}
\end{example}

In \cite[Theorem 13]{GSOj} we characterized gluings that are uniquely presented. If we apply this characterization to $S$ in Proposition \ref{prop: gluing-critical} with the notations in its proof, then $S$ is uniquely presented if and only if $\mathbb N\mathcal A$ is uniquely presented and  $\pm (c_{n+1}\mathbf a_{n+1}-c_i\mathbf a_i)\not\in S$ for all $i\in \{1,\ldots,n\}$. It is not hard to show that this is equivalent to $\mathbb N\mathcal A$ being uniquely presented and $c_{n+1}\mathbf a_{n+1}$ having a unique expression in terms of $\mathbf a_1,\ldots, \mathbf a_n$.

\section{On Northcott type ideals and set-theoretic complete intersections}

A classic problem in Commutative Algebra is to ask what is the least number of polynomials needed to determine the radical of a given ideal. This number is called the \textbf{arithmetic rank} of the ideal. Clearly, for a binomial ideal $I$, one has the following inequalities $$\mathrm{ht}(I) \leq \mathrm{ara}(I) \leq \mu(I),$$ where $\mathrm{ht}$, $\mathrm{ara}$ and $\mu$ denote the height, the arithmetic rank and the cardinal of a minimal system of generators of $I$ with smallest cardinality, respectively. 

Recall that \textbf{height of an ideal} $I$ is the infimum of the supremum of the lengths of all chains of prime ideals contained in $I$. 

\begin{definition}\label{def intersec}
An ideal $I$ of $\Bbbk[x_1, \ldots, x_n]$ is said to be:
\begin{itemize}
\item a \textbf{complete intersection} if $\mu(I) = \mathrm{ht}(I)$,
\item an \textbf{almost complete intersetion} if if $\mu(I) = \mathrm{ht}(I) + 1$,
\item a \textbf{set-theoretic complete intersection} if $\mathrm{ara}(I) = \mathrm{ht}(I)$, that is to say, if there exist $h$ polynomials $g_1, \ldots, g_h$ such that $\sqrt{I} = \sqrt{\langle g_1, \ldots, g_h \rangle}$, where $h = \mathrm{ht}(I)$.
\end{itemize}
\end{definition}

It is well-known that every minimal system of generators of a toric ideal has the same cardinal (see, for instance, \cite[Section 2]{OjVi3}). Now, if $I \subset \Bbbk[x_1, \ldots, x_n]$ is a critical ideal of Northcott type, by definition, there exist a matrix $\Phi$ and a vector $\mathbf{m}$ such that $I$ is minimally generated by $\{f_1, \ldots, f_{n-1}, D\}$ where $(f_1, \ldots, f_{n-1})^\top = \Phi\, \mathbf{m}^\top$ and $D = \det(\Phi)$. Therefore the ideal generated by $\{f_1, \ldots, f_{n-1}\}$ is a complete intersection and $I$ is an almost complete intersection.

\begin{proposition}\label{Prop rad1}
With the above notation, $$\langle f_1, \ldots, f_{n-1} \rangle = I \cap \langle \mathbf{m} \rangle,$$ where $\langle \mathbf{m} \rangle$ is the monomial ideal generated by the entries of $\mathbf{m}$.
\end{proposition}

\begin{proof}
By construction, it suffices to prove that $I \cap \langle \mathbf{m} \rangle \subseteq \langle f_1, \ldots, f_{n-1} \rangle$.
Let $f = \sum_{i=1}^{n-1} g_i f_i + g D \in I \cap \langle \mathbf{m} \rangle$. Since the first summand lies in $\langle f_1, \ldots, f_{n-1} \rangle \subseteq I \cap \langle \mathbf{m} \rangle$, we need to prove that $$\hbox{if } g D \in \langle \mathbf{m} \rangle \hbox{, then } g D \in \langle f_1, \ldots, f_{n-1} \rangle.$$ Moreover, since $\langle \mathbf{m} \rangle$ is a monomial ideal, $D = -x_n^{\sum_{i=1}^{n-1} u_{in}} + x_1^{u_{1n}} \cdots x_{n-1}^{u_{n-1\, n}}$ and no generator of $\langle \mathbf{m} \rangle$ depends on $x_n$ we may assume, without loss of generality, that $g$ is a minimal generator of $\langle \mathbf{m} \rangle$. Thus, let $g = x_j^{u_{\sigma(j) j}}$, where $\sigma = (1\, 2\, \cdots n-1) \in S_n$. For the sake of simplicity suppose $j = 1$, then 
\begin{align*}
g D = & x_1^{u_{21}} \big( -x_n^{\sum_{i=1}^{n-1} u_{in}} + x_1^{u_{1n}} \cdots x_{n-1}^{u_{n-1\, n}} \big) = x_2^{u_{2n}} \cdots x_{n-1}^{u_{n-1\, n}} f_1 + \\ & + x_2^{u_{2n}} \cdots x_{n-1}^{u_{n-1\, n} + u_{1\, n-1}} x_n^{u_{1n}} - x_1^{u_{21}} x_n^{\sum_{i=1}^{n-1} u_{in}} =  x_2^{u_{2n}} \cdots x_{n-1}^{u_{n-1\, n}} f_1 + \\ & +  x_2^{u_{2n}} \cdots x_{n-2}^{u_{n-2\, n}} x_n^{u_{1n}} f_{n-1} + x_2^{u_{2n}} \cdots x_{n-2}^{u_{n-2\, n} + u_{n-1\, n-2}} x_n^{u_{1n}+u_{n-1\, n}} - \\ & - x_1^{u_{21}} x_n^{\sum_{i=1}^{n-1} u_{in}} = \ldots = x_2^{u_{2n}} \cdots x_{n-1}^{u_{n-1\, n}} f_1 + x_2^{u_{2n}} \cdots x_{n-2}^{u_{n-2\, n}} x_n^{u_{1n}} f_{n-1} + \\ & +  \ldots + x_2^{u_{2n}} \big( x_n^{x_{1n} + \sum_{i=3}^{n-1} u_{in}} \big) f_3 + \big( x_n^{\sum_{i=1}^{n-1} u_{in}} \big) f_2
\end{align*}
and we are done. 
\end{proof}

Since $\langle \mathbf{m} \rangle$ is a complete intersection, Proposition \ref{Prop rad1} has the following interpretation in the context of liaison theory: the affine variety of a critical ideal of Northcott type is  directly CI-linked to the variety of $\langle \mathbf{m} \rangle$ (see \cite[Definition 5.1.2]{migliore}).

\begin{corollary}\label{Cor rad1}
With the above notation, $$\sqrt{\langle f_1, \ldots, f_{n-1} \rangle} = \sqrt{I} \cap \langle x_1, \ldots, x_{n-1} \rangle.$$
\end{corollary}

\begin{proof}
By Proposition \ref{Prop rad1}, $\langle f_1, \ldots, f_{n-1} \rangle = I \cap \langle \mathbf{m} \rangle$. Now, since $\sqrt{I \cap \langle \mathbf{m} \rangle} = \sqrt{I} \cap \sqrt{\langle \mathbf{m} \rangle}$ and 
$\sqrt{\langle \mathbf{m} \rangle} = \langle x_1, \ldots, x_{n-1} \rangle,$ our claim follows.
\end{proof}

\begin{lemma}\label{Lemma conj}
With the above notation, the reduction of $f_{n-1}^{c_1 \cdots c_{n-2}}$ modulo the ideal generated by $\{ f_1, \ldots, f_{n-2} \}$ is $$x_{n-1}^{\prod_{i=1}^{n-1} u_{\sigma(j)\, j}} (x_n^r + h),\ \text{for some}\ h \in \langle x_1, \ldots, x_{n-1} \rangle,$$ where $\sigma = (1\, 2\, \ldots n-1) \in S_n$ and $c_i = u_{ni}+u_{i+1\, i}$, $i\in\{ 1, \ldots, n-2\}$.
\end{lemma}

\begin{proof}
Set $c_{n-1} = u_{n\, n-1}+u_{1\, n-1}$ and $N = c_1 \cdots c_{n-2}$. Then $$f_{n-1}^N = \big( x_{n-1}^{c_{n-1}} - x_{n-2}^{u_{n-1\, n-2}} x_n^{u_{n-1\, n}}\big)^N = \sum_{j=0}^N (-1)^j \binom{N}{j}x_{n-1}^{(N-j) c_{n-1}} x_{n-2}^{j\, u_{n-1\, n-2}} x_n^{j\, u_{n-1\, n}}.$$ Let $S$ be the semigroup associated to the critical ideal of Northcott type generated by $\{f_1, \ldots, f_{n-1}, D\}$ and let $\prec$ be an $S-$graded reverse lexicographical order on $\Bbbk[x_1, \ldots, x_n]$ such that $x_1 \succ x_2 \succ \ldots \succ x_{n-2} \succ x_n \succ x_{n-1}$. Since $f_i$, $i \in\{ 1, \ldots, n-1\}$, are $S-$homogeneous, $\{f_1, \ldots, f_{n-2}\}$ forms a Gr\"obner basis with respect $\prec$. Now, by performing division by $\{f_{n-2}, \ldots, f_1\}$, whenever $j\, u_{j\, u_{n\, n-1}} \geq c_2$, we may substitute $x_{n-2}^{j\, u_{n-1\, n-2}}$ by a monomial that involves $x_n, x_{n-1}, x_{n-2}$ and $x_{n-3}$, such that the exponent of $x_2$ is less than $c_2$. If we repeat this process with $x_{n-3}, x_{n-4}, \ldots$ we obtain an equivalent expression, $g$, of $f_{n-1}^N$ modulo $\langle f_1, \ldots, f_{n-2} \rangle$ such that every monomial is multiple of $x_{n-1}$; in particular, the leading term of $g$ is $x_{n-1}^{\prod_{i=1}^{n-1} u_{\sigma(j)\, j}} x_n^r$ and its tail lies in $\langle x_1, \ldots, x_{n-1} \rangle$.
A careful computation shows that $P:=\prod_{i=1}^{n-1} u_{\sigma(j)\, j}$ is strictly smaller than the power of $x_{n-1}$ in any monomial of  $g-x_{n-1}^P$. So, we can write $g = x_{n-1}^P(x_n^r + h)$ for some $h \in  \langle x_1, \ldots, x_{n-1} \rangle$.
\end{proof}

The following result generalizes \cite[Theorem 3.1]{OCPV}.

\begin{theorem}\label{Th:STCI}
Every critical ideal of Northcott type is a set-theoretic complete intersection.
\end{theorem}

\begin{proof}
Let $I$ be a critical ideal of Northcott type. By construction, there exist a matrix $\Phi$ as in \eqref{ecPhi} and $\mathbf{m} = (x_1^{u_{21}}, \ldots, x_{n-2}^{u_{n-1\, n-2}}, x_{n-1}^{u_{1\, n-1}}),$ with $u_{ij} > 0$ for all $i,j,$ such that $I$ is generated by $f_1, \ldots, f_{n-1}$ and $D = \det(\Phi)$, where $(f_1, \ldots, f_{n-1})^\top = \Phi\, \mathbf{m}^\top$. Moreover, by Theorem \ref{th:gen-N}, there exists $\mathcal{A} \subset \mathbb{Z} \oplus T$, where $T$ is a finite commutative group, such that $\mathbb N \mathcal A \cap (-\mathbb N \mathcal A) = \{0\}$ and $I = I_\mathcal{A}$.
In particular, $(I:x_i) = I$ for every $i$, as it is stated in the proof of Theorem \ref{th:gen-N}.

By Lemma \ref{Lemma conj}, there exist $f = x_n^r + h$, with $r \geq 1$ and $h \in \langle x_1, \ldots, x_{n-1} \rangle$, such that $f_{n-1} \in \sqrt{\langle f_1, \ldots, f_{n-2}, f \rangle}$. Moreover, $f \in (\langle f_1, \ldots f_{n-1} \rangle : x_{n-1}^\infty) \subset (I: x_{n-1}^\infty) = I$. Thus, $$\sqrt{\langle f_1, \ldots, f_{n-1} \rangle} \subseteq  \sqrt{\langle f_1, \ldots, f_{n-2}, f \rangle} \subseteq \sqrt{I}.$$ Now, $$D^r - f^{\sum_{i=1}^{n-1} u_{in}} \in I \cap \langle x_1, \ldots, x_{n-1} \rangle \subseteq \sqrt{I} \cap \langle x_1, \ldots, x_{n-1} \rangle.$$ Therefore, by Corollary \ref{Cor rad1}, $$D^r - f^{\sum_{i=1}^{n-1} u_{in}} \in \sqrt{\langle f_1, \ldots, f_{n-1} \rangle} \subseteq  \sqrt{\langle f_1, \ldots, f_{n-2}, f \rangle},$$ which implies $D \in  \sqrt{\langle f_1, \ldots, f_{n-2}, f \rangle}$ and we are done.
\end{proof}

We end this paper by showing that, alike the generalization of the family (N), the families (gN1) and (gN2) are as well set-theoretically complete intersections.

\begin{proposition}\label{stci_glu}
Let $G$ be a commutative group, $\mathcal A = \{\mathbf{a}_1, \ldots, \mathbf{a}_n\} \subset G $ and $\mathbf{a}_{n+1} \in G$. If $I_\mathcal{A}$ is a set-theoretic complete intersection and $S=\mathbb N\mathcal A+\mathbb N\mathbf a_{n+1}$ is a gluing of $\mathbb N \mathcal{A}$ and $\mathbb N \mathbf{a}_{n+1}$, then $I_S$ is a set-theoretic complete intersection.
\end{proposition}

\begin{proof}
Let $h = \mathrm{ht}(I_\mathcal{A})$. By hypothesis, there exist $f_1, \ldots, f_h$ such that $\sqrt{f_1, \ldots, f_h} = \sqrt{I_\mathcal{A}}$. By construction, $I_S = I_\mathcal{A} + \langle f \rangle$ and  $\mathrm{ht}(I_S) = h+1$. Putting this together, we obtain that \begin{align*}\sqrt{I_S} = \sqrt{I_\mathcal{A} + \langle f \rangle} = \sqrt{\sqrt{I_\mathcal{A}} + \sqrt{\langle f \rangle}} = \sqrt{\sqrt{\langle f_1, \ldots, f_h \rangle} + \sqrt{\langle f \rangle}} = \sqrt{\langle f_1, \ldots, f_h, f \rangle}\end{align*} and we are done.
\end{proof}

\begin{corollary}
Let $G$ be a commutative group, $\mathcal A = \{\mathbf{a}_1, \ldots, \mathbf{a}_n\} \subset G $ and $\mathbf{a}_{n+1} \in G$.
If $\mathbb N \mathcal{A}$ is a monoid of Northcott type and $S=\mathbb N\mathcal A+\mathbb N\mathbf a_{n+1}$ is a gluing of $\mathbb N \mathcal{A}$ and $\mathbb N \mathbf{a}_{n+1}$, then $I_S$ is a set-theoretic complete intersection.
\end{corollary}

\begin{proof}
The proof is a straightforward consequence of Theorem \ref{Th:STCI} and Proposition \ref{stci_glu}.
\end{proof}

\textbf{Acknowledgments.} The authors would like to thank the referee for a careful reading and valuable comments on this paper.


\end{document}